\documentclass[11pt]{article}
\usepackage[margin=1in]{geometry}
\usepackage{amsthm,amsmath}
\usepackage{latexsym,amssymb,amsmath}

\usepackage{tikz}
\usetikzlibrary{positioning,fit,backgrounds,calc}
\usepackage{xcolor}

\usepackage{graphicx,float,color,fancybox,shapepar,setspace,hyperref}
\usepackage{subfigure}
\usepackage{pgf,tikz}
\usepackage[affil-it]{authblk}
\usepackage{indentfirst}
\usepackage[section]{placeins}
\hypersetup
{
colorlinks=true,
linkcolor=blue,
filecolor=blue,
urlcolor=blue,
citecolor=cyan,
}
\usetikzlibrary{arrows}
\newtheorem{theorem}{Theorem}
\newtheorem{lemma}{Lemma}
\newtheorem{corollary}{Corollary}

\newtheorem{conjecture}{Conjecture}

\newtheorem{definition}{Definition}

\newcounter{remarkcounter}

\def\dfn#1{{\sl #1}}

\def\~{\sim}

\def\less{\setminus}

\newcommand{\F}{\mathbb F}

\newcommand{\Rea}{\operatorname{Re}}

\begin{document}
\title{Odd covers for complete graphs and complete 3-graphs}
\author{Ting HUANG, Jiabao YANG, Yaojun CHEN\footnote{Corresponding author. Email: yaojunc@nju.edu.cn}\\
 \small{School of Mathematics, Nanjing University, Nanjing 210093, P.R. CHINA}}
\date{}
\maketitle

\begin{quote}
\noindent {\bf Abstract:}
The Graham-Pollak theorem says that one needs at least $n - 1$ complete bipartite graphs to cover each edge of a complete graph $K_{n}$ on $n$ vertices exactly once. The odd cover problem is a parity analogue which seeks the minimum number of complete bipartite graphs, denoted by $b_2(n)$, such that each edge of $ K_n $ is covered an odd number of times.
An odd cover of a complte 3-graph $K_n^{(3)}$ on $n$ vertices is a family of complete $3$-partite $3$-graphs such that every triple is covered an odd number of times.  Let $b_3(n)$ be the minimum size of such a family. The values of $b_2(n)$ and $b_3(n)$ are determined for some $n$ in several previous works. In this paper, we first determine the value of $b_2(n)$ for all $n$, which confirms a conjecture due to Buchanan et al. (JGT, 2026), and then show $b_3(n+1)=b_2(n)$ by which the value of $b_3(n)$ is determined for all $n$, that resolves a question posed by Leader and Tan (EJC, 2026).


\end{quote}

\section{Introduction}
Throughout this paper, we call complete bipartite graph by biclique. Let $K_n$ be the complete graph on $n$ vertices.
A classical theorem of Graham and Pollak~\cite{GrahamPollak1971,GrahamPollak1972} states that one needs at least $n - 1$ bicliques to cover every edge of $K_{n}$ exactly once. Given a finite simple graph $G$, let $\mathrm{bp}(G)$ denote the minimum number of bicliques whose edge sets partition $E(G)$. Thus Graham-Pollak states that $\mathrm{bp}(K_n) = n - 1$.

Babai and Frankl posed the following question in \cite{BabaiFrankl1992}: What is the minimum number of bicliques needed to cover every edge of $K_{n}$ an odd number of times? They called this the odd cover problem, generalized in \cite{beaudrap} as follows. Let $G$ be a finite simple graph. An \emph{odd cover} of $G$ is a collection of bicliques on subsets of the vertex set $V(G)$ which cover each edge of $G$ an odd number of times and each nonedge of $G$ an even number of times. 
An odd cover always exists for any graph $G$ since one trivial example is the collection of bicliques $\bigl(\{u\},\{v\}\bigr)$ defined on each edge $uv$ of $G$. In fact, a collection of bicliques is an odd cover of $G$ if and only if taking the symmetric difference over their edge sets recovers $E(G)$.
 Hence, it is worthwhile to determine the least number of bicliques forming an odd cover of $G$. Let $b_{2}(G)$ denote the minimum cardinality of such an odd cover of $G$. Because $K_n$ has no non-edges, so $b_2(K_n)$ is the smallest number of bicliques needed to cover every edge of $K_n$  an odd number of times. For convenience, we write $b_2(n) = b_2(K_n)$. In this language, Babai and Frankl~\cite{BabaiFrankl1992} asked for the value of $b_{2}(n)$.

The following figure gives one small example. Here $\triangle$ denotes symmetric difference of edge sets.

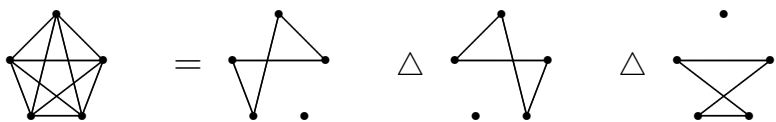
\begin{figure}[htbp]
\centering
\setlength{\unitlength}{0.28mm}
\begin{picture}(410,70)(0,-8)
\linethickness{0.55pt}

\qbezier(30,52)(19,41)(8,30)
\qbezier(30,52)(41,41)(52,30)
\qbezier(30,52)(24,28)(18,4)
\qbezier(30,52)(36,28)(42,4)
\qbezier(8,30)(30,30)(52,30)
\qbezier(8,30)(13,17)(18,4)
\qbezier(8,30)(25,17)(42,4)
\qbezier(52,30)(35,17)(18,4)
\qbezier(52,30)(47,17)(42,4)
\qbezier(18,4)(30,4)(42,4)
\put(30,52){\circle*{4}}
\put(8,30){\circle*{4}}
\put(52,30){\circle*{4}}
\put(18,4){\circle*{4}}
\put(42,4){\circle*{4}}

\put(92,27){\makebox(0,0){\Large $=$}}

\qbezier(135,52)(146,41)(157,30)
\qbezier(135,52)(129,28)(123,4)
\qbezier(113,30)(135,30)(157,30)
\qbezier(113,30)(118,17)(123,4)
\put(135,52){\circle*{4}}
\put(113,30){\circle*{4}}
\put(157,30){\circle*{4}}
\put(123,4){\circle*{4}}
\put(147,4){\circle*{4}}

\put(197,27){\makebox(0,0){\large $\triangle$}}

\qbezier(240,52)(229,41)(218,30)
\qbezier(240,52)(246,28)(252,4)
\qbezier(262,30)(240,30)(218,30)
\qbezier(262,30)(257,17)(252,4)
\put(240,52){\circle*{4}}
\put(218,30){\circle*{4}}
\put(262,30){\circle*{4}}
\put(228,4){\circle*{4}}
\put(252,4){\circle*{4}}

\put(302,27){\makebox(0,0){\large $\triangle$}}

\qbezier(323,30)(345,30)(367,30)
\qbezier(323,30)(340,17)(357,4)
\qbezier(333,4)(350,17)(367,30)
\qbezier(333,4)(345,4)(357,4)
\put(345,52){\circle*{4}}
\put(323,30){\circle*{4}}
\put(367,30){\circle*{4}}
\put(333,4){\circle*{4}}
\put(357,4){\circle*{4}}

\end{picture}
\caption{An odd cover of $K_5$. }
\label{fig:oddcoverK5}
\end{figure}

The problem of determining $b_2(G)$ was put forward by Niel de Beaudrap \cite{beaudrap}. Prior work contains alternate formulations equivalent to odd covers. The paper \cite{kaminski} introduced a procedure called \textit{bipartite subgraph complementation}, which complements the edges and nonedges between two disjoint subsets of vertices of the graph. Using this vocabulary, minimizing the cardinality of an odd cover of an $n$-vertex graph $G$ reduces to finding the least number of bipartite subgraph complementations that turn the $n$-vertex empty graph into $G$. This ties into the topic examined in \cite{bu1}: determining how many subgraph complementations, the operation of complementing the edge set of an induced subgraph, are necessary to generate $G$ from an $n$-vertex empty graph.

Buchanan et al. provided a general lower bound for $b_2(G)$
 in terms of $\mathbb{F}_2$-rank $r_2(A_G)$, where $A_G$ is the adjacency matrix of $G$ over $\F_2$.
\begin{theorem}[Buchanan et al.~\cite{BuchananCliftonCulverONeillRombachYin2023}]\label{thm:lower}
 For every graph $G$, $b_2(G)\ge r_2(A_G)/2$.
\end{theorem}

Finite-field constructions giving equality in some even cases go back to Radhakrishnan, Sen and Vishwanathan~\cite{RadhakrishnanSenVishwanathan2000}.  The following result gives the general range of possible values.

\begin{theorem}[Buchanan et al.~\cite{BuchananCliftonCulverONeillRombachYin2023}]\label{thm:range}
For every $n\ge2$,
\[
        \left\lceil \frac{n}{2}\right\rceil \le b_2(n)\le 
        \left\lceil \frac{n}{2}\right\rceil+1.
\]
Moreover, $b_2(n)=\lceil n/2\rceil$ whenever $n\equiv0,1,7\pmod8$.
\end{theorem}

In the same paper, Buchanan et al. formulated two conjectures concerning the value of $b_2(n)$, and the one for even $n$ was reformulated in \cite{BuchananCliftonCulverFranklNieOzekiRombachYin2026}. 

\begin{conjecture}[Buchanan et al.~\cite{BuchananCliftonCulverONeillRombachYin2023}]\label{conj:72}
Let $k\ge2$ be an integer. Then
$
        b_2(2k+1)=k+1.
$
\end{conjecture}

\begin{conjecture}[Buchanan et al.~\cite{BuchananCliftonCulverFranklNieOzekiRombachYin2026}]\label{conj:71}
For $k\ge2$ and $k\equiv2,3\pmod4$,
$b_2(2k)=k+1$.
    
\end{conjecture}

Very recently, Conjecture \ref{conj:72} was confirmed independently by Buchanan et al. \cite{BuchananCliftonCulverFranklNieOzekiRombachYin2026}, and Leader and Tan \cite{LeaderTan2026}.
\begin{theorem}[Buchanan et al.~\cite{BuchananCliftonCulverFranklNieOzekiRombachYin2026}; Leader and Tan~\cite{LeaderTan2026}]\label{thm:odd}
For every $k\ge 2$,
$
        b_2(2k+1)=k+1.
$

\end{theorem}
Clearly, Theorem \ref{thm:odd} determines the value of $b_2(n)$ for all odd $n$. For $n=2k$ being even,
the difficulty lies in when does $K_{2k}$ have an odd cover with exactly $k$ complete bipartite graphs.  Such an optimal cover will be called a \dfn{perfect odd cover}.
Several cases for even $n$ were already known.  For example,  $b_2(n)=n/2$ is known when $n\equiv0\pmod 8$~\cite{BuchananCliftonCulverONeillRombachYin2023}, when $n\equiv18\pmod {24}$~\cite{BuchananCliftonCulverFranklNieOzekiRombachYin2026}, and when $n=3^s-1$~\cite{LeaderTan2026}.  
Buchanan et al.~\cite{BuchananCliftonCulverONeillRombachYin2023} found $
        b_2(10)=6$ by computer-aided calculations, this implies
$K_{10}$ has no perfect odd cover. 
  
\vskip 2mm
In this paper, our first main result is to establish a sufficient and necessary condition for $K_{2k}$ having a perfect odd cover as below, which also confirms Conjecture \ref{conj:71}.

\begin{theorem}\label{thm:perfect}
Let $k\ge1$. Then $K_{2k}$ has a perfect odd cover if and only if
$k\equiv0,1\pmod4$ and $k\ne5$.

\end{theorem}
The proof of Theorem \ref{thm:perfect} has two new ingredients.  
 The first is a positive construction, see Subsection~\ref{deletion construction}.  Starting from the block construction of Buchanan et al.~\cite{BuchananCliftonCulverFranklNieOzekiRombachYin2026}  for $n\equiv18\pmod{24}$, we delete rows and columns from different blocks.  This non-symmetric deletion preserves exactly the parity conditions needed for a perfect odd cover and produces the missing cases $n\equiv10\pmod{24}$, except for $K_{10}$.
The second is a quadratic-form method, see Subsection~\ref{quadratic-form obstruction}.  It makes it impossible to construct any perfect odd cover of $K_{2k}$ when $k\equiv2,3\pmod4$, with no structural assumption on the cover.

From $
        b_2(10)=6$ and Theorems \ref{lem:blockknown}, \ref{thm:odd}, \ref{thm:perfect}, the value of $b_2(n)$ is completely determined for all $n$ as follows. 
\begin{theorem}\label{thm:mainb2}
For every $n\ge2$,
\[
 b_2(n)=
 \begin{cases}
 \lceil n/2\rceil, & n \hbox{ is odd},\\[1mm]
 n/2, & n\equiv 0\pmod 8,\\[1mm]
 n/2, & n\equiv 2\pmod 8\hbox{ and } n\ne 10,\\[1mm]
 n/2+1, & n\equiv 4,6\pmod 8,\\[1mm]
 6, & n=10.
 \end{cases}
\]
\end{theorem}

Now, let us turn to  the same problem for complete $r$-uniform hypergraphs for $r\ge 3$.  
Let $K_n^{(r)}$ denote the complete $r$-uniform hypergraph on $n$ vertices.  A \dfn{complete $r$-partite $r$-graph} is one whose vertex set consists of $r$ pairwise disjoint sets $A_1,...,A_r$ and edge set consists of those $ r $-sets that meet every $ A_i $. An odd cover of $K_n^{(r)}$ is a family of complete $r$-partite $r$-graphs such that every $r$-set of $K_n^{(r)}$ is covered an odd number of times.  Let $b_r(n)$ denote the minimum size of such a family.

For complete $3$-graphs, Leader and Tan proved the following result.

\begin{theorem}[Leader and Tan~\cite{LeaderTan2026}]\label{thm:LTb3}
For every even $n$, $b_3(n)=n/2$.  For every odd $n$, one has
\[
        b_3(n)\in \left\{ \frac{n-1}{2},\frac{n+1}{2}\right\}.
\]
Moreover, if $n$ is odd and either $n\equiv1\pmod8$ or $n$ is a power of $3$, then $b_3(n)=(n-1)/2$.
\end{theorem}
In the same paper, Leader and Tan wrote: ``It would be very interesting to determine what happens for other values of $n$.''  

\vskip 2mm
In this paper, our second main result is to determine the value of $b_3(n)$ for all $n$.

\begin{theorem}\label{thm:b3identity}
For every $n\ge2$,
$
        b_3(n+1)=b_2(n).
$
\end{theorem}

Now apply Theorem \ref{thm:b3identity} with $n-1$ in place of $n$, and then by Theorem~\ref{thm:mainb2}, we get the value of $b_3(n)$ for all $n$.

For $r\ge 3$, the link of a vertex $v$ in an $r$-graph is the $(r-1)$-graph consisting of the $(r-1)$-sets obtained by removing $v$ from
the $r$-sets that contain it. Leader and Tan~\cite{LeaderTan2026} observed that
by taking the link of a vertex, it is not difficult to show that $ b_r(n) \geq b_{r-1}(n-1) $. Indeed, given an odd cover of $ K_n^{(r)} $, let $ v $ be a vertex, and for each complete $ r $-partite $ r $-graph that contains $ v $ in the cover, we form a complete $ (r-1) $-partite $ (r-1) $-graph by removing the class $ A_i $ containing $ v $. These complete $ (r-1) $-partite $ (r-1) $-graphs form an odd cover of $ K_n^{(r-1)} $, as required.
In fact, we can further show that $ b_r(n) = b_{r-1}(n-1) $ whenever $r\ge 3$ is odd.
Apart from the odd cover problem, the exact partition analogue for complete $r$-graphs has also been studied extensively; see, for example, Leader, Mili\'cevi\'c and Tan~\cite{LeaderMilicevicTan2018}.

\section{The pairs construction}
 In this section, we first present a useful pairs construction from the literature~\cite{RadhakrishnanSenVishwanathan2000}. 
This provides a standard way to produce perfect odd covers of even complete graphs.

Set 
$V(K_{2k})=\{a_1, b_1, a_2, b_2, \dots, a_k, b_k\}$.
Let  
$(X_1, Y_1), \dots, (X_k, Y_k)$ be $ k $ bicliques.
Let $S=(s_{ij})$ be a $k\times k$ matrix with entries in $\{0,1,-1\}$, where
\[s_{ij} =
\begin{cases} 
0, & \text{neither } a_i \text{ nor } b_i \text{ is in the } j\text{-th biclique} (X_j, Y_j), \\
1, & a_i \in X_j,\ b_i \in Y_j, \\
-1, & a_i \in Y_j,\ b_i \in X_j.
\end{cases}\]


The following lemma can be extracted from Buchanan et al.~\cite{BuchananCliftonCulverFranklNieOzekiRombachYin2026}, and the original necessity result was given by Radhakrishnan, Sen, and Vishwanathan~\cite{RadhakrishnanSenVishwanathan2000}.
For the sake of completeness, we include a proof here.

\begin{lemma}
\label{lem:pairs-criterion}
The above matrix $S$ gives a \emph{perfect odd cover} of $K_{2k}$ if and only if the following three parity conditions hold:
\begin{enumerate}
    \item\label{cond:row-odd} Each row must contain an odd number of $ \pm 1 $'s;
    \item\label{cond:cross-sign-odd} For two distinct rows, there are an odd number of columns where one has $1$ and the other has $-1$;
    \item\label{cond:cross-same-odd} For two distinct rows, there are an odd number of columns where both have 1 or both have $-1$.
\end{enumerate}
\end{lemma}

\begin{proof}
Fix a row $i$. The edge $a_i b_i$ is covered in the $j$-th biclique if and only if $s_{ij} \neq 0$. Hence the number of times $a_i b_i$ is covered equals the number of nonzero entries in row $i$. For an odd cover this number must be odd, giving the condition \ref{cond:row-odd}.

Now take two distinct rows $i \neq l$. Consider four types of edges.

First, the edge $a_i a_l$ is covered in the $j$-th biclique if and only if $a_i$ and $a_l$ lie on opposite sides of the bipartition, i.e. one is in $X_j$ and the other in $Y_j$. By definition this is equivalent to $s_{ij}$ and $s_{lj}$ both being nonzero and having opposite signs, i.e. one is $1$ and the other is $-1$. Similarly, the edge $b_i b_l$ is covered exactly in the same columns of opposite signs. Therefore, for both $a_i a_l$ and $b_i b_l$ to be covered an odd number of times, the number of columns with nonzero entries of opposite signs must be odd. This gives the condition \ref{cond:cross-sign-odd}.

Second, the edge $a_i b_l$ or the edge $b_i a_l$ is covered in the $j$-th biclique if and only if $s_{ij}$ and $s_{lj}$ are both 1 or both $-1$. 
Hence, for both $a_i b_l$ and $b_i a_l$ to be covered an odd number of times, the number of columns where both have $1$ or both have $-1$ must be odd. This gives the condition \ref{cond:cross-same-odd}.

These edges comprise all edges of $K_{2k}$, so the three conditions are necessary and sufficient.
\end{proof}

\begin{definition}

A matrix $S=(s_{ij})\in\{0,1,-1\}^{k\times k}$ is called \dfn{admissible} if it satisfies the above three parity conditions.
\end{definition}

Therefore, by Lemma~\ref{lem:pairs-criterion}, if we want to construct a perfect odd cover of $K_{2k}$, we do not need to explicitly draw the $k$ bicliques on the $2k$ vertices; it suffices to construct a $k\times k$ admissible matrix. In other words, if there exists a $k\times k$ admissible matrix, then $K_{2k}$ has a perfect odd cover.

Buchanan et al.~\cite{BuchananCliftonCulverFranklNieOzekiRombachYin2026} gave the following $3m\times 3m$ matrix via the block matrix construction, and proved that this matrix is admissible whenever $m\equiv 3 \pmod 4$.

\begin{lemma}[Buchanan et al.~\cite{BuchananCliftonCulverFranklNieOzekiRombachYin2026}]\label{lem:blockknown}
Let
\[
        S_{3m}=
        \begin{pmatrix}
        A&C&O\\
        C&O&A\\
        O&A&C
        \end{pmatrix}.
\]
If $m\equiv3\pmod4$, then the block matrix is admissible.  Hence $K_{6m}$ has a perfect odd cover.
\end{lemma}

 Here $A$ is the $m\times m$ all-one matrix and $O$ is the $m\times m$ zero matrix.  Define an $m\times m$ matrix $C=(c_{ij})$ by $c_{ii}=0$ and, for $i<j$,
\[
        c_{ij}=\begin{cases}
        1, & j-i\hbox{ is odd},\\
        -1, & j-i\hbox{ is even},
        \end{cases}
        \qquad c_{ji}=-c_{ij}.
\]
For example, when $ m = 7 $,

\[
C = 
\begin{bmatrix}
0 & 1 & -1 & 1 & -1 & 1 & -1 \\
-1 & 0 & 1 & -1 & 1 & -1 & 1 \\
1 & -1 & 0 & 1 & -1 & 1 & -1 \\
-1 & 1 & -1 & 0 & 1 & -1 & 1 \\
1 & -1 & 1 & -1 & 0 & 1 & -1 \\
-1 & 1 & -1 & 1 & -1 & 0 & 1 \\
1 & -1 & 1 & -1 & 1 & -1 & 0
\end{bmatrix}.
\]

\section{Proof of Theorem~\ref{thm:perfect}}
\subsection{A parity-preserving deletion construction}\label{deletion construction}
In this subsection, we prove that $b_2(2k)=k$ for $k\equiv1\pmod4$, except for the known exceptional case $k=5$. It suffices to construct an admissible matrix. And our construction will start from $S_{3m}$ and perform row and column deletion operations as follows.

\begin{lemma}\label{lem:deletion}
Let $m\equiv3\pmod4$, and let $d$ be an integer such that
$
        d\equiv0\pmod4,~ 0\le d<m.
$
From the matrix $S_{3m}$ in Lemma~\ref{lem:blockknown}, delete the first $d$ rows of the second row block and the first $d$ columns of the first column block.  The resulting $(3m-d)\times(3m-d)$ matrix is admissible.
\end{lemma}

\begin{proof}
Let the resulting matrix be $T_{3m-d}$.  Since $S_{3m}$ is already admissible, it suffices to show that the deleted columns contribute an even number of same-sign nonzero columns and an even number of opposite-sign nonzero columns for any two retained rows. 
For convenience, denote the three blocks of rows by $R_1,R_2,R_3$ and the three blocks of columns by $C_1,C_2,C_3$ in $S_{3m}$. We delete the first $d$ columns of $C_1$ and the first $d$ rows of $R_2$ from $S_{3m}$. Consequently, all retained rows in $R_2$ have internal indices strictly greater than $d$.

First, we check each row of $T_{3m-d}$ still contain an odd number of nonzero entries.
A row from $R_1$ originally has $2m-1$ nonzero entries. The $d$ deleted columns all come from block $A$ (all one), leaving $2m - 1 - d$
nonzero entries. Since $d$ is even, $2m - 1 - d$ remains odd.
A retained row from $R_2$ originally has $2m-1$ nonzero entries. Its entries over the first $d$ deleted columns lie in block $C$; as the row index exceeds $d$, there are no diagonal zeros in these positions, so we remove exactly $d$ nonzero entries. The remaining count $2m-1-d$ is still odd.
Every row from $R_3$ has block $B$ (all zeros) in the first column block, so the deleted columns have no effect on its nonzero count, which stays at $2m-1$. Condition \ref{cond:row-odd} therefore holds.

Second, we analyze contributions from the deleted columns to pairs of rows.
We split into cases based on which row blocks the two retained rows belong to.

\text{Case 1: Both rows lie in $R_1$.}
Over the $d$ deleted columns, both rows take entry $1$. The same-sign contribution equals $d$, and the opposite-sign contribution equals $0$. Since $d\equiv 0\pmod{4}$, both values are even.

\text{Case 2: One row from $R_1$, one retained row from $R_2$.}
The $R_1$ row has entry $1$ across all deleted columns. Let the internal index of the $R_2$ row in block $C$ be $t>d$. For columns $1,2,\dots,d$ (all left of the diagonal), the definition of block $C$ states that the entry is $1$ if the column index shares parity with $t$, and $-1$ otherwise. Among the first $d$ integers, exactly $d/2$ share parity with $t$ and $d/2$ have opposite parity. Since $d\equiv 0\pmod{4}$, $d/2$ is even, so both the same-sign and opposite-sign contributions are even.

\text{Case 3: One row from $R_1$, one row from $R_3$.}
All entries of $R_3$ in the first column block lie in block $B$ (all zeros). No deleted column contains nonzero entries for both rows, so both same-sign and opposite-sign contributions equal $0$.

\text{Case 4: Both rows are retained rows from $R_2$.}
Let their internal indices be $p,q>d$. For any deleted column $c\le d$, we have $c<p,q$. The entry $C_{p,c}$ is determined solely by the parity of $p-c$, and $C_{q,c}$ solely by the parity of $q-c$. If $p,q$ have identical parity, the two rows carry identical signs across all $d$ deleted columns; if $p,q$ have opposite parity, their signs differ on every deleted column. The same-sign/opposite-sign contribution is thus either $d$ or $0$, both even.

\text{Case 5: Pairs involving $R_3$ where the second row is not from $R_1$.}
All entries of $R_3$ over deleted columns are zero, so all contributions remain $0$.

All cases confirm that the deleted columns supply an even number of same-sign nonzero columns and an even number of opposite-sign nonzero columns for any pair of retained rows. The original odd count properties from $S_{3m}$ are therefore preserved, so $T_{m,d}$ satisfies conditions \ref{cond:cross-sign-odd} and \ref{cond:cross-same-odd}. Hence $T_{3m-d}$ is admissible. 
\end{proof}

\begin{corollary}\label{cor:k1}
If $k\equiv 0,1\pmod4$ and $k\ne5$, then $K_{2k}$ has a perfect odd cover.
\end{corollary}

\begin{proof}
 If $k\equiv 0\pmod4$, it is known $b_2(2k)=k$ in~\cite{BuchananCliftonCulverONeillRombachYin2023} and then $K_{2k}$ has a perfect odd cover. We only consider the case for $k\equiv 1\pmod4$.

If $k=1$, then $K_2$ itself is one biclique.  If $k=13$, then $2k=26=3^3-1$, and then the result follows from the known
finite-field construction of Leader and Tan~\cite{LeaderTan2026}.

Now assume $k\ge9$, $k\ne13$, and $k\equiv1\pmod4$.  We shall write $k=3m-d$ with $m\equiv3\pmod4$, $d\equiv0\pmod4$, and $0\le d<m$.  
There are three cases modulo $12$.

If $k\equiv9\pmod{12}$, take
$m=k/3$ and $d=0$.

If $k\equiv5\pmod{12}$ and $k\ge17$, take
$m=(k+4)/3$ and $d=4$.

If $k\equiv1\pmod{12}$ and $k\ge25$, take
$m=(k+8)/3$ and $d=8$.

In each case $m\equiv3\pmod4$, $d\equiv0\pmod4$, $0\le d<m$, and $k=3m-d$.  The proof is complete.     
\end{proof}

\subsection{A quadratic-form method}\label{quadratic-form obstruction}
In this subsection we prove that a perfect odd cover cannot exist for half of the even
values.
\begin{theorem}\label{thm:obstruction}
If $K_{2k}$ has a perfect odd cover, then
$
        k\equiv0,1\pmod4.
$
Thus, if $k\equiv2,3\pmod4$, then $K_{2k}$ has no perfect odd cover.
\end{theorem}
\begin{proof}

Suppose, for a contradiction, that $K_{2k}$ has a perfect odd cover by $k$ bicliques
\[
        (L_1,R_1),(L_2,R_2),\ldots,(L_k,R_k).
\]
For a vertex $v$, define a vector
\[
 r_v=(\ell_1(v),r_1(v),\ell_2(v),r_2(v),\ldots,\ell_k(v),r_k(v))\in\F_2^{2k},
\]
where $\ell_i(v)=1$ if $v\in L_i$ and $0$ otherwise, and $r_i(v)=1$ if $v\in R_i$ and $0$ otherwise.  We work over $\F_2$, so all additions below are modulo $2$.

For $x=(x_1,x_2,x_3,x_4,\ldots,x_{2k-1},x_{2k})\in\F_2^{2k}$, define
\[
        q(x)=x_1x_2+x_3x_4+\cdots+x_{2k-1}x_{2k}=\sum_{i=1}^k x_{2i-1}x_{2i}.
\]
This is the quadratic form which checks, modulo $2$, how many coordinate pairs have both entries equal to $1$.  

For $x=(x_1,x_2,\ldots,x_{2k-1},x_{2k})\in\F_2^{2k}$ and $y=(y_1,y_2,\ldots,y_{2k-1},y_{2k})\in\F_2^{2k}$,
the associated bilinear form is
\[
        B(x,y)=q(x+y)+q(x)+q(y)
        =\sum_{i=1}^k(x_{2i-1}y_{2i}+x_{2i}y_{2i-1}).
\]
Since a vertex cannot lie in both sides of the same biclique, we deduce that $x_{2i-1}=0$ or $x_{2i}=0$ for all $1\leq i\leq 2k$.
It implies that every vertex vector satisfies
\begin{equation}\label{q(r)}
    q(r_v)=0.
\end{equation}
For two distinct vertices $u,v$, if the biclique $(L_i,R_i)$ covering the edge $uv$, then  
$x_{2i-1}y_{2i}+x_{2i}y_{2i-1}=1$. 
Hence the value $B(r_u,r_v)$ is exactly the parity of the number of bicliques covering the edge $uv$.  Because the cover is an odd cover, we have in $\F_2$ that
\begin{equation}\label{B(r,r)}
    B(r_u,r_v)=1 \qquad (u\ne v).              
\end{equation}
An easy observation shows that in $\F_2$, \begin{equation}\label{B(r,r)-equality}  B(r_v,r_v)=0.             \end{equation}

We claim that the $2k$ vectors $\{r_v: v\in V(K_{2k})\}$ form a basis of $\F_2^{2k}$.  Suppose
\[
        \sum_v \alpha_v r_v=0,
        \qquad \alpha_v\in\F_2.
\]
Fix a vertex $w$ and apply the linear functional $B(\cdot,r_w)$ to both sides.  Using (\ref{B(r,r)}) and (\ref{B(r,r)-equality}), we get
\[
        0=B\left(\sum_v \alpha_v r_v,r_w\right)=\sum_v\alpha_v B(r_v,r_w)=\sum_{v\ne w}\alpha_v.
\]
Let $T=\sum_v\alpha_v$.  Then the above equality says $T+\alpha_w=0$ in $\F_2$, hence $\alpha_w=T$.  This holds for every $w$.  Therefore
$
        T=\sum_v\alpha_v=2kT=0
$
in $\F_2$, since $2k$ is even.  Thus $T=0$ and all $\alpha_v=0$.  
It means that the vectors $r_v$ are linearly independent. 
Since there are $2k$ of them in a $2k$-dimensional space, they form a basis.

Now compute the Gauss sum
\[
        G(q)=\sum_{x\in\F_2^{2k}}(-1)^{q(x)}
\]
in two ways, where $G(q)\in \mathbb R$.

First use the original coordinate pairs. For any $x=(x_1,x_2,\ldots,x_{2k-1},x_{2k})$ of $\F_2^{2k}$
and take one pair of coordinates $(x_{2i-1},x_{2i})$, there are only four combinations: $(0,0),(0,1),(1,0),(1,1)$, then
\[
        \sum_{x_{2i-1},x_{2i}\in\F_2}(-1)^{x_{2i-1}x_{2i}}=1+1+1-1=2.
\]
There are $k$ independent coordinate pairs, so

\begin{align}\label{G(q)1}
     G(q)\!=\!\!\sum_{x\in\F_2^{2k}}(-1)^{q(x)}
     \!=\!\!\sum_{x\in\F_2^{2k}}\!\left(\prod_{i=1}^{k}(-1)^{x_{2i-1}x_{2i}}\right) 
     \!=\!\prod_{i=1}^{k} \left( \sum_{x_{2i-1},x_{2i}\in\F_2}\!(-1)^{x_{2i-1}x_{2i}} \right)
     \!=\!2^k.
\end{align}

Second use the basis $\{r_v: v\in V(K_{2k})\}$.  Every vector $x\in\F_2^{2k}$ can be written uniquely as $x=\sum_{v\in S}r_v$
for some subset $S$ of $V(K_{2k})$.  If $|S|=w$, then by (\ref{q(r)}) and (\ref{B(r,r)}),
\begin{align*}
    q(x)&=q\left(\sum_{v\in S}r_v \right)=\sum_{v\in S}q(r_v)+\sum_{\{u,v\}\subseteq S}B(r_u,r_v)=\sum_{\{u,v\}\subseteq S}B(r_u,r_v)=\binom{w}{2}\pmod2.
\end{align*}
Note that for each $w$, there are exactly $\binom{2k}{w}$ vectors in $\mathbb F_2^{2k}$. Hence
\begin{align*}\label{G(q)2}
     G(q)=\sum_{x\in\F_2^{2k}}(-1)^{q(x)} =\sum_{x\in\F_2^{2k}}(-1)^{\binom{w}{2}} =\sum_{w=0}^{2k}\binom{2k}{w}(-1)^{\binom w2}.
\end{align*}

We use $\Rea(z)$ to denote the real part of the complex number $z$.
Since the signs $(-1)^{\binom w2}$ repeat with period $4$ as $+,+,-,-,+,+,-,-,\ldots,$
it follows that $(-1)^{\binom w2}=\Rea((1-i)i^w)$.
Combining this with the binomial theorem, we obtain
\[
\begin{aligned}
\sum_{w=0}^{2k}\binom{2k}{w}(-1)^{\binom w2}
&=\Rea\left(\sum_{w=0}^{2k}\binom{2k}{w}(1-i)i^w\right)\\
&=\Rea\left((1-i)\sum_{w=0}^{2k}\binom{2k}{w}i^w\right)\\
&=\Rea\left((1-i)(1+i)^{2k}\right)\\
&=2^k\Rea\left((1-i)i^k\right).
\end{aligned}
\]
Therefore
\begin{equation}\label{G(q)3}
    G(q)=
        \begin{cases}
        2^k, & k\equiv0,1\pmod4,\\
        -2^k, & k\equiv2,3\pmod4.
        \end{cases} 
\end{equation}

Equations (\ref{G(q)1}) and (\ref{G(q)3}) are incompatible when $k\equiv2,3\pmod4$.  This contradiction proves the theorem.     
\end{proof}

\begin{proof}[Proof of Theorem \ref{thm:perfect}]
The conclusion follows from Corollary \ref{cor:k1} and Theorem \ref{thm:obstruction}.
\end{proof}


\section{Proof of Theorem \ref{thm:b3identity}}

We now prove $b_3(n+1)= b_2(n)$.
As mentioned in introduction, 
$ b_r(n) \geq b_{r-1}(n-1) $ for all $r\ge 3$. We only show that $b_3(n+1)\le b_2(n)$.



Let
$
        (X_1,Y_1),(X_2,Y_2),\ldots,(X_t,Y_t)
$
be an odd cover of $K_n$, where $t=b_2(n)$.  Add a new vertex $z$.  For every $i$, define
\[
        Z_i=(V(K_n)\less(X_i\cup Y_i))\cup\{z\}.
\]
Now form the complete $3$-partite $3$-graph with parts
$ (X_i, Y_i, Z_i)
$.
We claim that these $t$ complete $3$-partite $3$-graphs form an odd cover of $K_{n+1}^{(3)}$.

First consider a triple $\{a,b,z \}$.  It is covered by the $i$th $3$-partite graph if and only if the edge $ab$ is covered by the biclique $(X_i,Y_i)$.  Since the bicliques oddly cover $K_n$, the triple $\{a,b,z\}$ is covered oddly.

Now consider a triple $\{a,b,c\}\subseteq V(K_n)$. 
  For a fixed $i$, define
\[
\chi_i(uv) = 
\begin{cases} 
1, & \text{if } uv \text{ is covered by } (X_i, Y_i), \\
0, & \text{otherwise.}
\end{cases}
\]
The triple $\{a,b,c\}$ is covered by $(X_i,Y_i,Z_i)$ if and only if among $a,b,c$, exactly one lies in $X_i$, exactly one in $Y_i$, and exactly one in $Z_i$. This is equivalent to saying that the number of edges among $\{ab,ac,bc\}$
that are crossed by $(X_i,Y_i)$ is odd. In formulae,
\[
\mathbf{1}_{\{a,b,c\}\in (X_i,Y_i,Z_i)} \equiv \chi_i(ab)+\chi_i(ac)+\chi_i(bc) \pmod{2}.
\]
Summing over all $i$, we get
\[
\sum_i \mathbf{1}_{\{a,b,c\}\in (X_i,Y_i,Z_i)}
\equiv
\sum_i \chi_i(ab)+\sum_i \chi_i(ac)+\sum_i \chi_i(bc) \pmod{2}.
\]
Since the original biclique cover covers every edge an odd number of times, the right‐hand side equals
$
1+1+1 \equiv 1 \pmod{2}
$.
Hence the triple $\{a,b,c\}$ is also covered an odd number of times. That is every triple not containing $z$ is also covered oddly.  This proves $b_3(n+1)\le b_2(n)$.


\section*{Acknowledgement}
\noindent This research is supported by National Key R\&D Program of China under grant number 2024YFA1013900 and NSFC under grant number 12471327.
\section*{Declaration}
	
\noindent$\textbf{Conflict~of~interest}$
The authors declare that they have no known competing financial interests or personal relationships that could have appeared to influence the work reported in this paper.
\vskip 2mm	
\noindent$\textbf{Data~availability}$
No data was used for the research described in the article.

\end{document}